\documentclass[11pt,a4paper]{article}
\usepackage{amssymb,amsmath,amsthm,amsfonts}
\usepackage[top=2.5cm, bottom=3cm]{geometry}
\usepackage{stmaryrd}

\usepackage{graphicx}
\usepackage{caption}
\usepackage[labelfont=it,textfont=normalfont]{subcaption}

\usepackage{color}
\usepackage{hyperref}
\hypersetup{
    bookmarks=true,
    unicode=false,
    pdftoolbar=true,
    pdfmenubar=true,
    pdffitwindow=false,
    pdfstartview={FitH},
    pdfnewwindow=true,
    pagebackref=false,
    colorlinks=true,
    linkcolor=black,
    citecolor=red,
    filecolor=magenta,
    urlcolor=black
}

\usepackage{algpseudocode}
\usepackage[section]{algorithm}
\usepackage{tikz}

\theoremstyle{plain}
\newtheorem{theorem}{Theorem}[section]
\newtheorem{lemma}{Lemma}[section]
\newtheorem{proposition}{Proposition}[section]
\newtheorem{observation}{Observation}[section]

\theoremstyle{definition}
\newtheorem{definition}{Definition}[section]
\newtheorem{example}{Example}[section]
\newtheorem{code}{Algorithm}[section]

\theoremstyle{remark}

\newtheorem{remark}{Remark}[section]
\newtheorem{corollary}{Corollary}[section]

\numberwithin{equation}{section}

\def\N{\mathbb{N}}

\def\R{\mathbb{R}}
\def\C{\mathbb{C}}

\def\K{\mathbb{K}}
\def\I{\mathbb{I}}

\def\D{\mathcal{D}}
\def\O{\mathcal{L}}

\def\e{\varepsilon}

\def\Re{\mathrm{Re}}
\def\Im{\mathrm{Im}}
\def\dim{\mathrm{dim}}

\newcommand\norm[1]{\interleave#1\interleave}
\newcommand\x[1]{\mathsf{#1}}
\newcommand\ex[1]{\langle #1 \rangle}
\newcommand\co[2][]{\,\rangle\hspace{-1pt}#2\hspace{-1pt}\langle_{#1}\,}

\newcommand{\interval}[1]{\llbracket #1 \rrbracket}

\def\Item$#1${\item $\displaystyle#1$   \hfill\refstepcounter{equation}(\theequation)}

\renewcommand{\theenumi}{\arabic{enumi})}
\renewcommand{\labelenumi}{\textit{\theenumi}}

\definecolor{jszary}{gray}{0.8}

\title{\bf Strict localization of eigenvectors and eigenvalues}

\author{
\L{}ukasz Struski\thanks{Corresponding author}\; and Jacek Tabor\\[0.3em]
\small Jagiellonian University,\\[-0.8ex]
\small Faculty of Mathematics and Computer Science,\\[-0.8ex]
\small \L{}ojasiewicza 6, 30-348 Krak\'ow, Poland\\[-0.5ex]
\small e-mail: \href{mailto:struski@ii.uj.edu.pl}{\sf struski@ii.uj.edu.pl}, \href{mailto:tabor@ii.uj.edu.pl}{\sf tabor@ii.uj.edu.pl}
}

\date{}

\begin{document}
 \maketitle

\begin{abstract}
In this article we show and implement a simple and efficient method to strictly locate eigenvectors and eigenvalues of a given matrix, based on the modified cone condition. As a consequence we can also effectively localize zeros of complex polynomials.

\bigskip

\noindent\small {\bf Key words and phrases:} eigenvectors, eigenvalues, cone condition, spectrum,  interval arithmetic, zeros of polynomials.\\[0.2em]
\small {\bf 2010 Mathematics Subject Classification:} 65F15. 
\end{abstract}

\medskip
           
        %
        %
        %
\section{Introduction}
 
 \subsection{Motivation}
 
Determination of the eigenvalues and eigenvectors of a matrix is important in many areas of science, for example in  web ranking \cite{google1,google2}, computer graphics and visualization  \cite{twarz}, quantum mechanics \cite{qm}, statistics  \cite{pca1, pca}, medicine, communications, construction vibration analysis \cite{capd, vibration}. 

One of the most important numerical methods designed to calculate the eigenvalues and eigenvectors of matrix $A$ is the power method \cite{Com, jacobi}. It is used to determine a maximum module eigenvalue of $A$ and the corresponding eigenvector $\x{v}$. The limit of the product $\frac{A^n\x{w}}{\|A^n\x{w}\|}$, where $\x{w}$ is a randomly chosen element, is the vector corresponding to the largest eigenvalue. The eigenvalue can be calculated from the Rayleigh quotient $\lambda = \frac{\langle A\x{v}, \x{v}\rangle}{\langle\x{v},\x{v}\rangle}$.
The most common method of solving the full eigenvalue problem, i.e.\ finding all the eigenvalues and corresponding eigenvectors, is the QR method \cite{QQR, QR}. 

There are methods for locating eigenvalues such as Gerschgorin theorem \cite{G} from 1931. This theorem allows to strictly locate the position of the eigenvalues of the matrix with real or complex coefficients. However, it does not allow to localize the eigenvectors.
 
With the growing importance of these concepts there is a need to look for new methods of localizing simultaneously eigenvalues and eigenvectors. We do not know strict and efficient methods for locating eigenvectors of real or complex matrices, so we want to fill this gap. Our aim is to create and analyze a new method of strict location eigenvectors and eigenvalues with the use of interval arithmetic\footnote{For language C++ one can use libraries such as `boost` \cite{boost} or `CAPD` \cite{capd}.}. 
Using the fact that the polynomial $W(\x{x})=\x{x}^n+a_{n-1}\x{x}^{n-1} + \ldots +a_1\x{x}+a_0$ of the $n$-th degree  equals the determinant of the matrix
\[
\begin{bmatrix}
0 & 1 & 0 & \ldots & 0           \\[0.3em]
0 & 0 & 1 & \ddots & \vdots          \\[0.3em]
\vdots & \vdots & \ddots & 1 & 0  \\[0.3em]
0 & 0 & \ldots & 0 & 1  \\[0.3em]
-a_0 & -a_1 & -a_2 & \ldots & -a_{n-1}
\end{bmatrix},
\]
we show that we can effectively localize zeros of real or complex polynomials (for non-strict methods see \cite{R1,R2}). 

The content of this paper can be described as follows. In the following short subsection we present an algorithm which  is the final outcome of our theory. In the Section 2 we introduce notation and present the simple properties of our theory, which  is based the concept of cone condition. The notion of cone condition originally appeared in the late 60's in the works of Alekseev, Anosov, Moser and Sinai \cite{Alek}. These techniques are used in the examination of differential equations \cite{CZ, Hs, KT, Nh, Z}. In Section 3 we show the basic properties of operations. In the following section we present the main result of this article Corollary \ref{x_1}, which allow as a consequence obtain strictly localize the position of the eigenvector corresponding to the given eigenvalue in a small cone. In the last section we present a simple algorithm and some numerical applications of the theory on a few basic examples. In particular for a random matrix of size $5\times 5$ with randomly generated values from the set $[-1,1]$ we obtain the eigenvectors and eigenvalues with typical accuracy $\e\approx 10^{-11}$ (see Example \ref{ex})\footnote{In our program we use variables of type double. One could decrease the error by using numbers of arbitrary precision.}.

\subsection{Basic Algorithm}

In this subsection we present a basic idea how our method can be used (more detailed explanation is given in last section). This algorithm determines the accuracy for numerical approximation of the eigenvectors and the corresponding eigenvalues. Now we introduce the following notation for clarity and simplicity of the presentation.

A matrix $A$ of dimensions $N\times N$ with elements $a_{ij}\in\K$ is denoted by 
\[
A:=[a_{ij}]_{1\leq i,j\leq N}\in M_N(\K).
\] 
For $K,J\subset\{1,\ldots,N\}$ by $A_{[K,J]}:=[a_{kj}]_{k \in K, j \in J}$ we denote the sub-matrix with rows $k\in K$ and columns $j\in J$. We put $m:n=\{m,\ldots,n\}$ for $m,n\in\{1,\ldots,N\}$. For $k\in\{1,\ldots,N\}$ by $\neq\!k$ we understand the set
\(
\{1,\ldots, N\}\setminus\{k\}.
\)
 The identity matrix is denoted by $I$.

By $\interval{\x{x}}$ we understand interval representation of $\x{x}\in\R$ that is an interval with representable ends such that $\x{x}\in\interval{\x{x}}$. For complex number $\x{z}\in\C$ we put $\interval{\x{z}}:=\interval{\Re(\x{z})}+\interval{\Im(\x{z})}i$, and for a matrix $A=[a_{ij}]_{1\leq i,j\leq N}$, $\interval{A}$ denotes the matrix $[\interval{a_{ij}}]_{1\leq i,j\leq N}$. We refer the reader to \cite{Moore} for more information concerning interval arithmetic.

As a sup norm $\|A\|_{\infty}$ we take the maximum sum of absolute values of the elements in rows of matrix $A$, that is 
\[
\|A\| _{\infty}= \|[a_{ij}]\|_{\infty} := \max\limits_{1\leq i\leq N}\sum_{j=1}^{N}\|a_{ij}\|_{\infty}.
\]

We are given a matrix $A=[a_{ij}]_{1\leq i,j\leq N}\in M_N(\C)$ with pairwise disjoint single eigenvalues, $\lambda_i\neq\lambda_j$ for $i\neq j$ and $i,j\in\{1, \dots, N\}$. We assume that we have calculated the numerical approximation of eigenvectors $\x{\tilde{x}_1}, \ldots, \x{\tilde{x}_N}$ and eigenvalues\footnote{To calculate the approximation of eigenvectors and eigenvalues one can use any  numerical method.} $\tilde{\lambda}_1,\ldots,\tilde{\lambda}_N$ of the matrix $A$.
To locate the exact position of $k$-th eigenvector $\x{x_k}$ of $A$ one can use the following algorithm.

\begin{code}\label{code}
\begin{enumerate}
\setlength{\itemsep}{3pt}
  \setlength{\parskip}{0pt}
  \setlength{\parsep}{0pt}
  \item Put  $P:=[\x{\tilde{x}_1}, \ldots, \x{\tilde{x}_N}]$,\label{p:1}
  \item calculate interval hull $\interval{P^{-1}}$ of inverse of $P$ (we call it the inverse interval matrix of $P$)\footnote{By $\interval{P^{-1}}$ we denote such an interval matrix that $Q^{-1} \in \interval{P^{-1}}$ for every $Q \in \interval{P}$.},\label{p:2}
  \item calculate the interval matrix $J=\interval{P^{-1}}\cdot\interval{A}\cdot \interval{P}$,\label{p:3}
  \item seek a constant $\e>0$ as small as possible so that the matrix \label{p:4}
\end{enumerate}
\begin{equation}\label{B^e}
B^{\e} = 
\begin{bmatrix}
 \frac{1}{\e} & 0 & \ldots  & 0 &  &  &  &  \\
0 & \ddots & \ddots  & \vdots &  & 0 &  & \\
\vdots & \ddots & \frac{1}{\e} & 0 &  &  &  & \\
0 & \ldots & 0 & 1 & 0 & \ldots & 0\\
&  &  &   0 & \frac{1}{\e} & \ddots & \vdots\\
& 0 &  & \vdots &\ddots & \ddots & 0 \\
&  &  & 0 & \ldots & 0 & \frac{1}{\e} 
\end{bmatrix}
\cdot ( J -\tilde{\lambda}_k\cdot I)\cdot
\begin{bmatrix}
 \e & 0 & \ldots  & 0 &  &  &  &  \\
0 & \ddots & \ddots  & \vdots &  & 0 &  & \\
\vdots & \ddots & \e & 0 &  &  &  & \\
0 & \ldots & 0 & 1 & 0 & \ldots & 0\\
&  &  &   0 & \e & \ddots & \vdots\\
& 0 &  & \vdots &\ddots & \ddots & 0 \\
&  &  & 0 & \ldots & 0 & \e
\end{bmatrix},
\end{equation}
where the number $1$ is at the intersection of the k-th row and k-th column, satisfies the following condition
  \begin{align*}
    \|B^{\e}_{[k,1:N]}\|_{\infty} & < \|(B^{\e}_{[\neq k,\neq k]})^{-1}\|_{\infty}^{-1} - \|B^{\e}_{[\neq k,k]}\|_{\infty}.
  \end{align*}
Then the $k$-th eigenvector of $A$ satisfies
\[
\x{x_k}\in \interval{\x{\tilde{x}_k}} +\interval{-\e,\e}\cdot \sum_{i\neq k}^{N} \interval{\Re(\x{\tilde{x}_i})}+\interval{-\e,\e}\cdot \sum_{i\neq k}^{N} \interval{\Im(\x{\tilde{x}_i})}i.
\]
\end{code}
\noindent We estimate the eigenvalue $\lambda_k$ from the Rayleigh coefficient
\[
\lambda_k \in \frac{\langle A\x{x_k}, \x{x_k}\rangle}{\langle\x{x_k},\x{x_k}\rangle}.
\]
We show more precise details in the last section of this paper.

\begin{remark}
Our method can be modified to work with multiple single eigenvalues (see Proposition \ref{dim}). However, it does not work with multiple eigenvalues of the form \eqref{bl} but this is a typical problem of most algorithms calculating the eigenvalues and eigenvectors.
\begin{equation}\label{bl}
\begin{bmatrix}
\lambda & 1  & 0  & \cdots & 0  \\
0  & \lambda  & \ddots & \ddots & \vdots \\
\vdots  & \ddots & \ddots & 1 & 0  \\
\vdots  & \;  & \ddots & \lambda    &   1\\
0  & \cdots  & \cdots & 0 & \lambda   \\
\end{bmatrix},
\end{equation}
\end{remark}

        %
        %
        %
\section{Cones and Dominating Maps}
 
In this section we modify the concept of cones from \cite{KT}, which will allow us to locate the  eigenvectors of the matrix.

\begin{definition}
Let $E$ be a finite-dimensional Banach space with semi-norms $\co[]{\cdot}$ (we call it {\em contracting}), $\ex{\cdot}$ (which we call {\em expanding}) such that
\[
\norm{\x{x}} := \max(\co[]{\x{x}}, \ex{\x{x}})
 \] 
 defines an equivalent norm on $E$.
 By the {\em r-norm} for $r>0$ on  the cone-space $E$ we take
\[
   \norm{\x{x}}_r:=\max(\co[]{\x{x}}, r\cdot\ex{\x{x}}).
\]
\end{definition}

\begin{definition}\label{D1}
Let $E$ be a cone-space. We define the {\em $r$-contracting cone} in $E$ by
\begin{align*}
   \co[r]{E} & :=\{\x{x} \in E: \;\co[]{\x{x}} \geq r\ex{\x{x}}\},\\
   \intertext{and the {\em $r$-expanding cone} in $E$ by}
   \ex{E}_r & :=\{\x{x} \in E:\; \co[]{\x{x}} \leq r\ex{\x{x}} \}.
\end{align*}
\end{definition}
\noindent Note that 
\begin{equation}\label{E}
E= \co[r]{E}\cup\ex{E}_r.
\end{equation}
In the same way we define $r$-contracting cone and $r$-expanding cone in subspace $E$.
If $r=1$ we will omit the subscript $r$, in particular we speak of contracting cone.
We introduce the scaling by $r$ of semi-norms to have better control over size of the cones (see Figure \ref{rys}), and which consequently will allow us to better locate the eigenvectors.
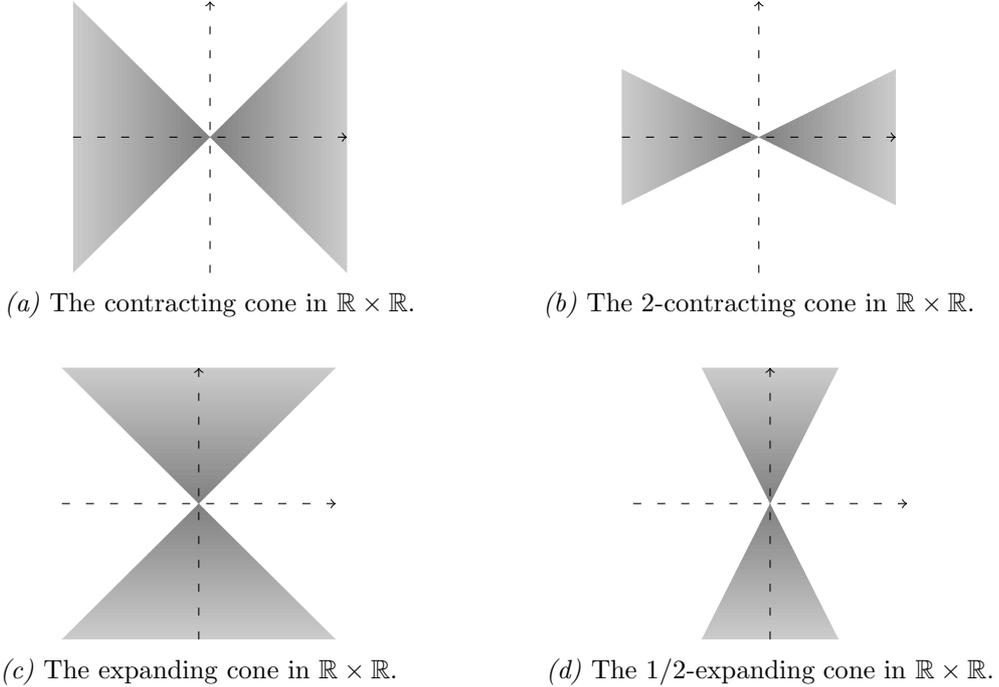
\begin{figure}[h!]
  \centering
        \begin{subfigure}[h!]{0.43\textwidth}
                \centering
		\begin{tikzpicture}[scale=1.2]
		\shade[left color=jszary, right color= gray] (-1.5,1.5) -- (0,0) -- (-1.5,-1.5) -- cycle;
		  \shade[right color=jszary, left color= gray] (1.5,1.5) -- (0,0) -- (1.5,-1.5) -- cycle;
		\draw[loosely dashed, ->] (-1.5,0) -- (1.5,0) coordinate (x axis);
		\draw[loosely dashed, ->] (0,-1.5) -- (0,1.5) coordinate (y axis);
		 \end{tikzpicture}
                \caption{The contracting cone in $\R\times\R$.}
        \end{subfigure}
        \qquad
        \begin{subfigure}[h!]{0.43\textwidth}
                \centering
                \begin{tikzpicture}[scale=1.2]
   		\shade[left color=jszary, right color= gray] (-1.5,0.75) -- (0,0) -- (-1.5,-0.75) -- cycle;
		  \shade[right color=jszary, left color= gray] (1.5,0.75) -- (0,0) -- (1.5,-0.75) -- cycle;
		\draw[loosely dashed, ->] (-1.5,0) -- (1.5,0) coordinate (x axis);
		\draw[loosely dashed, ->] (0,-1.5) -- (0,1.5) coordinate (y axis);
		 \end{tikzpicture}
                \caption{The $2$-contracting cone in $\R\times\R$.}
        \end{subfigure}
        \\[1.5em]
         \begin{subfigure}[h!]{0.45\textwidth}
                \centering
                \begin{tikzpicture}[scale=1.2]
   		\shade[top color= jszary, bottom color=gray] (-1.5,1.5) -- (0,0) -- (1.5,1.5) -- cycle;
		  \shade[top color=gray, bottom color=jszary] (-1.5,-1.5) -- (0,0) -- (1.5,-1.5) -- cycle;
		\draw[loosely dashed, ->] (-1.5,0) -- (1.5,0) coordinate (x axis);
		\draw[loosely dashed, ->] (0,-1.5) -- (0,1.5) coordinate (y axis);
		 \end{tikzpicture}
                \caption{ The expanding cone in $\R\times\R$.}
        \end{subfigure}
        \qquad
        \begin{subfigure}[h!]{0.45\textwidth}
                \centering
                \begin{tikzpicture}[scale=1.2]
   		\shade[top color= jszary, bottom color=gray] (-0.75,1.5) -- (0,0) -- (0.75,1.5) -- cycle;
		  \shade[top color=gray, bottom color=jszary] (-0.75,-1.5) -- (0,0) -- (0.75,-1.5) -- cycle;
		\draw[loosely dashed, ->] (-1.5,0) -- (1.5,0) coordinate (x axis);
		\draw[loosely dashed, ->] (0,-1.5) -- (0,1.5) coordinate (y axis);
		 \end{tikzpicture}
                \caption{The $1/2$-expanding cone in $\R\times\R$.}
        \end{subfigure}
        \caption{The cones in the cone-space $\R\times\R$.}\label{rys}
\end{figure}

For  a product $E = E_1\times E_2$ we introduce cone-space $(E, \co[]{\cdot}, \ex{\cdot})$ where we put
\[
 \co[]{\x{x}}:=\|x_1\|, \quad \ex{\x{x}}:=\|x_2\|\quad\text{ for }\;\x{x}=(x_1,x_2)\in E_1\times E_2.
\]
Analogically, we define cone-space for direct sum $E = E_1\oplus E_2$.

In our main result, Corollary \ref{x_1}, the following proposition will play a crucial role.

\begin{proposition}\label{dim}
Let $E=E_1\times E_2$ be cone-space such that $\dim E_1=n$, $\dim E_2=m$ and let $r>0$ be given. Assume that we have direct sum decomposition $E=V_1\oplus V_2$ such that
\[
V_1\subset\co[r]{E}\quad\text{and}\quad V_2\subset\ex{E}_r.
\]
Then
\[
\dim V_1=n\quad\text{and}\quad\dim V_2=m.
\]
\end{proposition}
\begin{proof}
First we show that $\dim V_1\leq n$. For an indirect proof, assume that $\dim V_1>n$. Then there exist linearly independent vectors $\x{v_1},\ldots,\x{v_{n+1}}\in V_1$. Obviously $\x{v_i}=(w_i, z_i)$ for $i\in\{1,\ldots,n+1\}$ and uniques $w_i\in E_1$, $z_i\in E_2$. Since $w_1,\ldots,w_{n+1}\in E_1$ and $\dim E_1=n$ there exist a set of $n+1$ scalars, $\alpha_1, \ldots,\alpha_{n+1}$, not all zero, such that
\[
\alpha_1 w_1+\ldots+\alpha_{n+1}w_{n+1}=0.
\]
Note that
\[
z:=\alpha_1 z_1+\ldots+\alpha_{n+1}z_{n+1}\neq 0,
\]
because otherwise vectors $\x{v_1},\ldots,\x{v_{n+1}}$ would not be linearly independent.
Consequently we obtain
\[
(0,z)=\left(\sum_{i=1}^{n+1}\alpha_i w_i, \sum_{i=1}^{n+1}\alpha_i z_i\right)\in V_1\subset\co[r]{E},
\]
and thus $r\|z\|\leq\|0\|$, which implicate $z=0$. We get a contradiction with the fact the sequence of vectors $\x{v_1},\ldots,\x{v_{n+1}}$ is linearly independent.

The proof that $\dim V_2\leq m$ is analogous. Finally, since $\dim E=n+m$ and $\dim V_1\leq n$, $\dim V_2\leq m$ we obtain
\[
\dim V_1=n,\quad\text{and}\quad\dim V_2=m.
\]
\end{proof}

By an {\em operator} we mean a linear mapping between cone-spaces $E$ and $F$. We denote the space of all operators by $\O(E,F)$. If $F=E$, we denote $\O(E,E)$ by $\O(E)$.

\medskip

Let $A\in\O(E,F)$. We define
\begin{alignat}{2}
   \co[r]{A} & := & \inf & \{R \in \R_+ \, | \,  \norm{A\x{x}}_r \leqslant R\norm{\x{x}}_r \text{ for all } \x{x}\in E: A\x{x} \in \co[r]{F}\}, \label{rate_1}\\[0.3em]
   \ex{A}_r & :=\; &\sup &\{ R \in \R_+ \, | \,  \norm{A\x{x}}_r \geqslant R\norm{\x{x}}_r \text{ for all }  \x{x} \in E: \x{x} \in \ex{E}_r\}. \label{rate_2}
\end{alignat}

Let $\tilde{E}\subset E$, $\tilde{F}\subset F$ be subspaces and let $A\in\O(E,F)$ such that $A(\tilde{E})\subset\tilde{F}$. We define
\begin{alignat*}{2}
   \co[r]{A|_{\tilde{E}}} & := & \inf & \{R \in \R_+ \, | \,  \norm{A\x{x}}_r \leqslant R\norm{\x{x}}_r \text{ for all } \x{x}\in\tilde{E}: A\x{x} \in \co[r]{\tilde{F}}\},\\[0.3em]
   \ex{A|_{\tilde{E}}}_r & :=\; &\sup &\{ R \in \R_+ \, | \,  \norm{A\x{x}}_r \geqslant R\norm{\x{x}}_r \text{ for all }  \x{x} \in\tilde{E}: \x{x} \in \ex{\tilde{E}}_r\}.
\end{alignat*}

\begin{definition}\label{def:1}
We say that $A\in\O(E,F)$ is {\em $r$-dominating}, if
   \[
   \co[r]{A} < \ex{A}_r.
   \]
By $\D_r(E,F)$ we denote the set of all $A\in\O(E,F)$ which are $r$-dominating. If $F=E$, we denote the space $\D_r(E,E)$ by $\D_r(E)$.
\end{definition}

\begin{observation}\label{ob:1}
Let $\tilde{E}\subset E$, $\tilde{F}\subset F$ be subspaces and let $A\in\O(E,F)$ such that $A(\tilde{E})\subset\tilde{F}$. We have
\[
\co[r]{A|_{\tilde{E}}}\leq\co[r]{A}\quad\text{ and }\quad\ex{A}_r\leq\ex{A|_{\tilde{E}}}_r.
\]
Moreover, if $A\in\D_r(E,F)$ then $A\in\D_r(\tilde{E},\tilde{F})$.
\end{observation}
\begin{proof}
It is consequence of \eqref{rate_1}, \eqref{rate_2} and Definition \ref{def:1}.
\end{proof}

\begin{theorem}\label{inter}
Let $A\in \D_r(E,F)$ and let  $\x{v}\in E$ be arbitrary. Then
\begin{align*}
  \x{v}\in \ex{E}_r & \implies A\x{v}\in \ex{F}_r ,\\[0.3em]
  A\x{v}\in \co[r]{F} & \implies \x{v}\in \co[r]{E}.
\end{align*}
\end{theorem}

\begin{proof}
The proof is a simple modification of the proof of \cite[Proposition 2.1]{KT}.
\end{proof}

\begin{proposition}\label{domin}
Let $A\in \D_r(F,G)$ and $B\in \D_r(E,F)$. Then $A\circ B\in \D_r(E,G)$ and
\begin{equation}\label{star}
  \co[r]{A\circ B}\leq \co[r]{A}\cdot\co[r]{B},\; \ex{A\circ B}_r \geq \ex{A}_r\cdot\ex{B}_r.
\end{equation}
\end{proposition}
\begin{proof}
To prove the first inequality from \eqref{star}, let $\x{x}\in E$ and $A\circ B(\x{x})\in \co[r]{G}$. From \eqref{rate_1} and Theorem \ref{inter} we know that $B\x{x}\in \co[r]{F}$. We have
\[
  \norm{A\circ B(\x{x})}_r\leq\co[]{A} \norm{B\x{x}}_r\leq\co[]{A}\co[]{B}\norm{\x{x}}_r.
\]

Hence
\[
  \co[]{A\circ B}\leq \co[]{A}\cdot\co[]{B}.
\]
Using \eqref{rate_2} and Theorem \ref{inter}, we obtain the second inequality from \eqref{star}.

As a simple consequence of \eqref{star} we obtain $A\circ B\in \D_r(E,G)$.
\end{proof}

        %
        %
        %

\section{Operator Norms}

In this section we show how given an operator $A$ to estimate $\co[r]{A}$, $\ex{A}_r$.

Consider two cone-spaces $E = E_1\times E_2$  and $F = F_1\times F_2$. 
Let $A\colon E\to F$ be an operator given in the matrix  form by
\[
 A = \begin{bmatrix}
       A_{11} &A_{12}           \\[0.3em]
       A_{21} & A_{22}
     \end{bmatrix}.
\]
 By 
\begin{equation}\label{norm}
\norm{A} _{r}:= \max\big(\|A_{11}\| + \frac{1}{r}\|A_{12}\|, r\|A_{21}\| + \|A_{22}\|\big)
\end{equation}
we define the $r$-norm of operator $A$. Observe that it satisfies
\[
\norm{A\x{x}} _{r}\leq \norm{A} _{r}\cdot \norm{\x{x}} _r \quad \text{ for }\; \x{x}\in E.
\]
Note that in general it is not,  if $E_1$ is not one dimensional, the classical operator norm for $\norm{\cdot}_{r}$.

\begin{observation}\label{OB}
Let $r\in (0,\infty)$, $A\in\O(E_1\times E_2,F_1\times F_2)$ We put
\[
 A = \begin{bmatrix}
       A_{11} &A_{12}           \\[0.3em]
       A_{21} & A_{22}
     \end{bmatrix}
     \quad \text{and }\quad
 R = \begin{bmatrix}
       I &0           \\[0.3em]
       0 & rI
     \end{bmatrix}.
 \]
Then
\begin{align*}
\norm{\x{x}}_r  &=  \norm{R\x{x}} \quad\text{ for }\; \x{x}\in E,\\
\norm{A}_r &=  \norm{RAR^{-1}}.
\end{align*}
\end{observation}

\begin{theorem}\label{Tw}
Let $A=[A_{ij}]_{1\leq i,j\leq 2}\in\O_r(E_1\times E_2,F_1\times F_2)$.
\renewcommand{\theenumi}{\roman{enumi}}
\renewcommand{\labelenumi}{\theenumi)}
\begin{enumerate}
\item We have
\[
\co[r]{A} \leq \| A_{11}\|+\frac{1}{r}\|A_{12}\|.
\]
\item Additionally, if $A_{22}$ is invertible, then
\[
\ex{A}_r \geq \|A_{22}^{-1}\|^{-1} - r\|A_{21}\|.
\]
\end{enumerate}
\end{theorem}
\begin{proof}We show the proof for $r=1$. For arbitrary $r\in(0,\infty)$ the proof  needs a simple modification, see Observation \ref{OB}.

For the proof of the first inequality, we take $\x{x}=(x_1,x_2)\in E_1\times E_2$ such that $A\x{x}\in\co[]{F}$. From Definition \ref{D1} we have 
\begin{equation}\label{eq:a}
\|A_{11}x_1+A_{12}x_2\|\geq \|A_{21}x_1+A_{22}x_2\|,
\end{equation}
 therefore
\begin{align*}
\norm{A\x{x}} &= \max(\|A_{11}x_1+A_{12}x_2\|, \|A_{21}x_1+A_{22}x_2\|)\\
&\overset{\eqref{eq:a}}{=}\|A_{11}x_1+A_{12}x_2\|\leq \{\|A_{11}\|\cdot \|x_1\| + \|A_{12}\|\cdot \|x_2\|\}\\
&\leq (\|A_{11}\| + \|A_{12}\|)\cdot \norm{\x{x}}.
 \end{align*}
 
For the proof of the second inequality, suppose that $\x{x}=(x_1,x_2)\in\ex{E}$, where $x_1\in E_1$, $x_2\in E_2$. Then
\begin{equation}\label{eq:b}
\|x_1\|\leq\|x_2\|=\norm{\x{x}}.
\end{equation}
 
We know that
\begin{equation}\label{E1}
\|A_{22}x_2\|\geq \|A_{22}^{-1}\|^{-1}\|x_2\| \geq 0.
\end{equation}
Finally, we obtain
\begin{align*}
\norm{A\x{x}} &\geq \|A_{21}x_1+A_{22}x_2\|\geq\|A_{22}x_2\| - \|A_{21}x_1\|\\
&\overset{\eqref{E1}}{\geq}\|A_{22}^{-1}\|^{-1}\|x_2\| -  \|A_{21}\|\|x_1\|\overset{\eqref{eq:b}}{\geq}\left(\|A_{22}^{-1}\|^{-1} -  \|A_{21}\|\right)\cdot\norm{\x{x}}.
\end{align*} 
\end{proof}

\begin{example}
Let us verify that the matrix $A\in\O(\R\times\R,\R\times\R)$
\[
 A =
 \begin{bmatrix}
  2 & 1.5 \\
  1 & 5
 \end{bmatrix}
\] 
is dominating.

By Theorem \ref{Tw} we have
\[
\co[]{A} \leq 3.5  < 4 \leq\ex{A},
\]
and therefore $A$ is dominating.
\end{example}

        %
        %
        %

\section{The Main Results}

In this section we show the main results of the paper, concerns the strict location of the eigenspace.

We denote the spectrum of the operator $A$ by $\sigma(A):=\{\lambda\in\C\; :\; A-\lambda I \text{ is singular}\}$.

\begin{lemma}\label{lem}
Let $A\in\D_r(E)$. Then
\begin{equation}\label{eq:1}
\lambda\in\sigma(A) \implies |\lambda|\subset [0,\co[r]{A}]\cup [\ex{A}_r,\infty).
\end{equation}
Moreover $[0,\co[r]{A}]\cap [\ex{A}_r,\infty)=\emptyset$.
\end{lemma}
\begin{proof}
Since $A\in\D_r(E)$ we get $[0,\co[r]{A}]\cap [\ex{A}_r,\infty)=\emptyset$.

Now we show implication \eqref{eq:1}. Let $\lambda$ be an eigenvalue of $A$ and let $\x{x}\in E$ be a corresponding eigenvector. By \eqref{E} we know that $\x{x}\in\co[r]{E}\cup\ex{E}_r$. We consider two cases. First suppose that $\x{x}\in\co[r]{E}$. Since $\x{x}$ is an eigenvector, $A\x{x} = \lambda\x{x}$, and thus $A\x{x}\in\co[r]{E}$. By \eqref{rate_1} we get
\[
|\lambda|\leq\co[r]{A}.
\]
Now suppose that $\x{x}\in\ex{E}_r$. By \eqref{rate_2} we get
\[
|\lambda|\geq\ex{A}_r,
\]
which complete the proof.
\end{proof}

Let $E$ be a vector space over the field $\K$ and let operator $A\colon E\to E$ be given.  One can easy deduce from the theorem of Jordan, see \cite[Appendix to Chapter 4]{irwin} for the general case, that if $\sigma = \sigma_1\cup\sigma_2$ (for  $\K=\R$ we assume additional that $\bar{\sigma}_1=\sigma_1$ and $\bar{\sigma}_2=\sigma_2$) then there is a unique direct sum decomposition $E=E_{\sigma_1}\oplus E_{\sigma_2}$ such that $A(E_{\sigma_1})\subset E_{\sigma_1}$, $A(E_{\sigma_2})\subset E_{\sigma_2}$ and $\sigma(A|_{E_{\sigma_1}})=\sigma_1$, $\sigma(A|_{E_{\sigma_2}})=\sigma_2$. For $c>0$ we define
\[
E_{\leq c}:=E_{\{\lambda\; :\; |\lambda|\leq c\}}\quad\text{and}\quad E_{\geq c}:=E_{\{\lambda\; :\; |\lambda|\geq c\}}.
\]

\begin{theorem}\label{Tw1}
Let $E$ be a finite-dimensional cone-space and let $A\in\D_r(E)$. Then there is a direct sum decomposition 
\[
E=E_{\leq\co[r]{A}}\oplus E_{\geq\ex{A}_r}\quad\text{ and }\quad E_{\leq\co[r]{A}} \subset  \co[r]{E},\; E_{\geq\ex{A}_r} \subset \ex{E}_r.\label{a1}
\]
\end{theorem}
\begin{proof}
For the clarity  of presentation we consider only the case $r=1$ and we omit the subscript $r$.

From Lemma \ref{lem} and Jordan Theorem we obtain a direct sum decomposition
\[
E=E_{\leq\co[]{A}}\oplus E_{\geq\ex{A}},
\]
such that 
\[
\sigma(A|_{E_{\leq\co[]{A}}})=[0,\co[]{A}]\quad\text{and}\quad \sigma(A|_{E_{\geq\ex{A}}})=[\ex{A},\infty).
\]

Now we show $E_{\leq\co[]{A}} \subset  \co[]{E}$. Consider an arbitrary $\x{x}\in E_{\leq\co[]{A}}$. The case when $\x{x}=0$ is obvious. Assume that $\x{x}\neq 0$. Without loss of generality  we can assume that $\|\x{x}\| = 1$.
For an indirect proof, assume that $\x{x}\notin\co[]{E}$. Then $\x{x}\in\ex{E}$ (by \ref{E}). Let $\e>0$ be arbitrary. From the fact that $\x{x}\in E_{\leq\co[]{A}}$, we know that
\[
\limsup\limits_{n\rightarrow +\infty}\sqrt[n]{\norm{A^n \x{x}}}\leq\co[]{A},
\]
and thus there exists $m\in\N$ such that for all $n\in\N$ 
\[
n\geq m \Rightarrow \sqrt[n]{\norm{A^n \x{x}}}\leq \co[]{A}+\e.
\]

Since $\x{x}\in\ex{E}$ and from Theorem \ref{inter} we obtain
\[
\x{x}\in\ex{E} \Rightarrow A\x{x}\in \ex{E} \Rightarrow  \cdots  \Rightarrow A^n\x{x}\in\ex{E}.
\]
Using \eqref{rate_2} we get
\begin{align*}
\norm{A\x{x}}&\geq \ex{A} \norm{\x{x}}, \\
\norm{A^2\x{x}} = \norm{A(A\x{x})}&\geq \ex{A} \norm{A\x{x}}\geq \ex{A}^2 \norm{\x{x}}, \\ 
&\;\, \vdots\\
\norm{A^n\x{x}}&\geq \ex{A}^n \norm{\x{x}}. 
\end{align*}
Finally we have
\[
\ex{A} = \sqrt[n]{\ex{A}^n} \leq \sqrt[n]{\norm{A^n \x{x}}} \leq \co[]{A}+\e.
\]
Since $\e$ was arbitrary, we get a contradiction with the fact that $A$ is $r$-dominating.

Analogously to the proof of the second conclusion, assume that $\x{x}\in E_{\geq\ex{A}}$ and $\x{x}\notin\ex{E}$.  Then  $\x{x}\in\co[]{E}$. Since $\sigma(A|_{E_{\geq\ex{A}}})=\sigma_{{\geq\ex{A}}}:=\{\lambda\; :\; |\lambda|\geq\ex{A}\}$ and $0\notin\sigma_{{\geq\ex{A}}}$ we know that $A|_{E_{\geq\ex{A}}}\colon E_{\geq\ex{A}}\to E_{\geq\ex{A}}$ is invertible.
Let $\e>0$ be arbitrary. Using the fact that $\x{x}\in  E_{\geq\ex{A}}$ we know that
\[
 \limsup\limits_{n\rightarrow +\infty}\sqrt[n]{\norm{A|_{ E_{\geq\ex{A}}}^{-n} \x{x}}}\leq\ex{A}^{-1},
 \]
and thus there exists $m\in\N$ such that for all $n\in\N$
\begin{equation}\label{e1}
n\geq m \Rightarrow\sqrt[n]{\norm{A|_{ E_{\geq\ex{A}}}^{-n} \x{x}}}\leq \ex{A}^{-1}+\e.
\end{equation}
From Observation \ref{ob:1} and Theorem \ref{inter} we get
\[
\x{x}\in\co[r]{ E_{\geq\ex{A}}} \Rightarrow A|_{ E_{\geq\ex{A}}}^{-1}\x{x}\in \co[]{ E_{\geq\ex{A}}} \Rightarrow  \cdots  \Rightarrow A|_{ E_{\geq\ex{A}}}^{-n}\x{x}\in\co[]{ E_{\geq\ex{A}}},
\]
and from \eqref{rate_1} we have
\begin{align*}
\norm{\x{x}}&\leq \co[]{A|_{ E_{\geq\ex{A}}}} \norm{A|_{ E_{\geq\ex{A}}}^{-1}\x{x}}, \\
\norm{A|_{ E_{\geq\ex{A}}}^{-1}\x{x}} &\leq \co[]{A|_{ E_{\geq\ex{A}}}} \norm{A|_{ E_{\geq\ex{A}}}^{-2}\x{x}}, \\ 
&\;\, \vdots\\
\norm{A|_{ E_{\geq\ex{A}}}^{-n+1}\x{x}}&\leq \co[]{A|_{ E_{\geq\ex{A}}}} \norm{A|_{ E_{\geq\ex{A}}}^{-n}\x{x}}.
\end{align*}
Hence
\begin{equation}\label{e2}
\norm{\x{x}}\leq (\co[]{A|_{ E_{\geq\ex{A}}}})^n\norm{A|_{ E_{\geq\ex{A}}}^{-n}\x{x}}.
\end{equation}

Finally from Observation \ref{ob:1} and \eqref{e1}, \eqref{e2} we obtain
\[
\co[]{A}\geq\co[r]{A|_{ E_{\geq\ex{A}}}} = \sqrt[n]{(\co[]{A|_{ E_{\geq\ex{A}}}})^{n}}\geq\sqrt[n]{\frac{1}{\norm{A|_{ E_{\geq\ex{A}}}^{-n} \x{x}}}}\geq\frac{1}{\ex{A}^{-1}+\e}=\ex{A}\cdot \frac{1}{1+\e\cdot\ex{A}},
\]
which get a contradiction with the fact that $A$ is $r$-dominating.
\end{proof}

By $\I$ we denote interval $\interval{-1,1}$. For $\e>0$ we put $B_{\C}(0,\e):=\{\x{z}\in\C\;:\; |\x{z}|\leq\e\}$.

\begin{corollary}\label{x_1}
Let $\e>0$ and $N\in\N$. Assume that an operator $A\in\D_r(\K\times\K^{N-1})$ for  $r = 1/\e$ is given.
\renewcommand{\theenumi}{\roman{enumi}}
\renewcommand{\labelenumi}{\theenumi)}
\begin{enumerate}
\item $\K=\R$. Then there exist unique eigenvalue $\lambda$ of $A$ such that $|\lambda|\leq\co[r]{A}$ and the eigenspace is one-dimensional space. The unique (module rescaling) eigenvector $\x{x}$ corresponding to the eigenvalue $\lambda$ satisfies
\[
\x{x}\in (1,0,\ldots,0)^T +\e\cdot(0,\I,\ldots,\I)^T.
\]
\item $\K=\C$. Then there exist unique eigenvalue $\lambda$ of $A$ such that $|\lambda|\leq\co[r]{A}$ and the eigenspace is one-dimensional space. The unique (module rescaling) eigenvector $\x{x}$ corresponding to the eigenvalue $\lambda$ satisfies
\[
\x{x}\in(1,0,\ldots,0)^T+\{0\}\times B_{\C}(0,\e)^{N-1}\subset(1,0,\ldots,0)^T+\e\cdot(0,\I,\ldots,\I)^T + \e\cdot(0,\I,\ldots,\I)^Ti.
\]
\end{enumerate}
\end{corollary}
\begin{proof}
It is a direct consequence of Theorem \ref{Tw1} and Proposition \ref{dim}.
\end{proof}

\begin{example}
Consider the following square matrix
\[
A=
\begin{bmatrix}
1 & 0.4 & 0.5 \\[0.3em]
0.4  & 4 & 0.4  \\[0.3em]
0.5 & 0.4 & 8
 \end{bmatrix}.
\]
We find the localization of $\x{x_1}$ eigenvector corresponding to the eigenvalue of smallest module. Since $A$ satisfies  for $r = 1.1$
\[
\co[r]{A}\leq 1.81(81) < 3.24 \leq \ex{A}_{r},
\]
we can use Theorem \ref{Tw1}. We get
\begin{equation}\label{equ:2}
  \x{x_1}\in \{(x_1, x_2, x_3)\; :\; \max\{|x_2|, |x_3|\}\leq \dfrac{10}{11}|x_1|\}.
\end{equation}

Finding the roots of the characteristic polynomial of $A$, which is equal to
\begin{align*}
W(\lambda) & = -\lambda^3+13\lambda^2-43.43\lambda+29.72 \\[0.3em]
& = -(\lambda-4)(\lambda-\frac{9-\sqrt{51.28}}{2})(\lambda-\frac{9+\sqrt{51.28}}{2}).
\end{align*}
we get
\begin{alignat*}{2}
&\x{x_1}=\left( \dfrac{\lambda_1-7.5}{\lambda_1-0.5}x, \dfrac{0.8\lambda_1-3.2}{(\lambda_1-4)(\lambda_1-0.5)}x, x\right)\; && \text{ for } \; x\in\R,\\[0.3em]
&\x{x_1}\approx\left( - 15.686641 x, 1.9070447x, x\right)\; &&\text{ for } \; x\in\R,
\end{alignat*}
Note that $\x{x_1}$ satisfies \eqref{equ:2}.
\end{example}

        %
        %
        %
\section{Applications}

In this section we show  the numerical applications of our theory. For the purposes of this article, we use the programming language c++ with `Boost` libraries \cite{boost} (version 1.51.0), which contains the interval arithmetic, and software library for numerical linear algebra `Linear Algebra PACKage` \cite{lapack} (revision 1348). Alternatively, one could use numerical computing environment `Matlab` with an interval package. In the first section  we presented a simple algorithm (see Algorithm \ref{code}), which strictly localizes the position of particular eigenvalues and eigenvectors of a matrix.

Now we discuss in more details the steps from the Algorithm \ref{code}. First we determine the numerical approximation of eigenvectors $\x{\tilde{x}_1}, \ldots, \x{\tilde{x}_N}$ and corresponding eigenvalues $\tilde{\lambda}_1,\ldots,\tilde{\lambda}_N$ of the matrix $A\colon\C^N\to\C^N$. In the step \ref{p:1} of this algorithm we create matrix $P$ of numerical approximate eigenvectors $\x{\tilde{x}_1}, \ldots, \x{\tilde{x}_N}$.
For the step \ref{p:2} we calculate the inverse matrix of the interval matrix which is composed of the approximate eigenvectors. The  inverse matrix we can determine by the method of Gauss-Jordan elimination.
In the next step we calculate the interval matrix 
\[
J=\interval{P^{-1}}\cdot\interval{A}\cdot\interval{P}.
\]
The most important stage of our method is the step \ref{p:4}. In this step we want to take advantage of Corollary \ref{x_1}, but first we need to check whether our matrix $J$ satisfies the assumptions of this corollary (is it $r$-dominating). We want to localize the $k$-th eigenvector of a general matrix. We create matrices $B^{\e}$ (see \eqref{B^e}) and we check whether the following inequality is satisfied (see Theorem \ref{Tw}):
\[
\co[]{B^{\e}} < \ex{B^{\e}}.
\]
If we find an $\e$ which satisfies the above condition for the matrix $B^{\e}$, then from Corollary \ref{x_1} we know that the $k$-th eigenvector of $J$ satisfies
\[
(1,0,\ldots,0)^T+\e\cdot(0,\I,\ldots,\I)^T + \e\cdot(0,\I,\ldots,\I)^Ti,
\]
where $\I=\interval{-1,1}$. Since $J=\interval{P^{-1}}\cdot\interval{A}\cdot\interval{P}$ we obtain the $k$-th eigenvector of $A$ multiplying the $k$-th eigenvector of $J$ by interval matrix $\interval{P}$. Finally we have
\[
\x{x_k}\in \interval{\x{\tilde{x}_k}} +\interval{-\e,\e}\cdot \sum_{i\neq k}^{N} \interval{\Re(\x{\tilde{x}_i})}+\interval{-\e,\e}\cdot \sum_{i\neq k}^{N} \interval{\Im(\x{\tilde{x}_i})}i.
\]
The eigenvalue $\lambda_k$ corresponding to the eigenvector $\x{x_k}$ we estimate from
\begin{equation}\label{eigenvalue}
\lambda_k \in \frac{\langle A\x{x_k}, \x{x_k}\rangle}{\langle\x{x_k},\x{x_k}\rangle}.
\end{equation}

Now we show some examples where we used our method.

\begin{example}\label{ex}
Consider $500$ random matrices of size $5\times 5$ with randomly generated values from the set $[-1,1]$. For each such matrix we launched our program to find accurate $\e$ location of all eigenvectors. We obtained $\e\in[10^{-15},10^{-5}]$ and for the results we have the fist-quantile to equal $6\cdot10^{-13}$, the median to equal $10^{-11}$ and the third-quantile to equal $8\cdot10^{-10}$. 
\end{example}

\begin{example}
Let $A$ be given in the matrix form
\[
A=
\begin{bmatrix}
-2 & 0 & 3 & 1 & -1 & 3 & -3 & -2 & 4 & 1 \\[0.3em]
0 & 3 & -1 & 4 & 2 & 2 & -3 & 0 & -4 & 0 \\[0.3em]
0 & 0 & -3 & -4 & 2 & -2 & 0 & -3 & -1 & 1 \\[0.3em]
-4 & -1 & 4 & 1 & -1 & 2 & 4 & 1 & 2 & 0 \\[0.3em]
3 & -1 & 4 & 0 & 4 & 3 & -2 & 0 & 1 & 3 \\[0.3em]
4 & -1 & 1 & 2 & 1 & -4 & 2 & -2 & -4 & -2 \\[0.3em]
0 & 4 & 1 & -1 & -2 & -4 & -2 & 4 & 1 & -1 \\[0.3em]
-3 & -4 & 2 & 0 & 0 & 0 & 0 & 0 & 4 & 0 \\[0.3em]
4 & -4 & 2 & 0 & -1 & 0 & -2 & -4 & 4 & 0 \\[0.3em]
-4 & -3 & 4 & 4 & 0 & 4 & -3 & 3 & 3 & -1
 \end{bmatrix}.
 \]
Our program found the following set of $\e\in[10^{-10},10^{-8}]$.

Due to the large size of eigenvectors of $A$ we write only first of them
\begin{alignat*}{2}
\x{x_1}\in &(&0.0032301+&[32, 51]\cdot 10^{-9}+( - 0.1909691+[-14, -10]\cdot 10^{-9})i,\\
&&0.0032301+&[20, 62]\cdot 10^{-9}+\;\; (0.1909691+[\quad\!\;9,\quad\! 16]\cdot 10^{-9})i,\\
&&0.1312400+&[61, 85]\cdot 10^{-9}+[-34, 34]\cdot 10^{-10}i,\\
&&0.2755005+&[42, 60]\cdot 10^{-9}+[-28, 28]\cdot 10^{-10}i,\\
&&0.1849394+&[66, 98]\cdot 10^{-9}+(- 0.0068575+[-54, -45]\cdot 10^{-9})i,\\
&&0.1849394+&[66, 98]\cdot 10^{-9}+\;\; (0.0068575+[\quad\!45,\quad\! 54]\cdot 10^{-9})i,\\
&&0.3754802+&[56, 77]\cdot 10^{-9}+\;\;(0.0266195+[\quad\!36,\quad\! 44]\cdot 10^{-9})i,\\
&&0.3754802+&[56, 77]\cdot 10^{-9}+(- 0.0266195+[-43, -36]\cdot 10^{-9})i,\\
&&0.1456179+&[89, 94]\cdot 10^{-9}+[-54, 54]\cdot 10^{-11}i,\\
&-&0.0358730+&[-7, 0]\cdot 10^{-9}+[-78, 78]\cdot 10^{-11}i ),
\end{alignat*}

The eigenvalue corresponding to the eigenvector $\x{x_1}$ satisfies
\[
\lambda_1\in 5.56625+[46, 81]\cdot10^{-7} + (3.1629\;+[69, 72]\cdot10^{-6})i.
\]
\end{example}

\begin{example}\label{roots}
Consider the polynomial
\[
W(\x{x})=\x{x}^5+(5-i)\x{x}^4-7i\x{x}^2+(2+4i)\x{x}-8.
\]
To find all the roots of polynomial $W(\x{x})$ we estimate the eigenvalues of the matrix 
\[
A=
\begin{bmatrix}
0 & 1 & 0 & 0 & 0 &\\[0.3em]  
0 & 0 & 1 & 0 & 0 & \\[0.3em]  
0 & 0 & 0 & 1 & 0 & \\[0.3em]  
0 & 0 & 0 & 0 & 1 & \\[0.3em]
8 & -2-4i & 7i & 0 & -5+i
\end{bmatrix}.
\]
From our program we obtain
\[
\begin{array}{cl}
\lambda_1\in &-5.1189735+[-66, -07]\cdot 10^{-9} + 1.2610393+[39, 67]\cdot 10^{-9}i,\\
\lambda_2\in &\quad\!0.384\;\,+[\quad\!37,\quad\! 42]\cdot 10^{-5} + 1.2215+[\quad\!45,\quad\! 97]\cdot 10^{-6}i,\\
\lambda_3\in &\quad\!0.9572+[\quad\!53,\quad\! 90]\cdot 10^{-6} + 0.1374+[\quad\!16,\quad\! 42]\cdot 10^{-6}i,\\
\lambda_4\in &-1.0805+[-80, -10]\cdot 10^{-6}\, - 0.6647+[-95, -42]\cdot 10^{-6}i,\\
\lambda_5\in &-0.1421+[-68, -30]\cdot 10^{-6}\, - 0.9552+[-97, -45]\cdot 10^{-6}i.
\end{array}
\]
\end{example}

        %
        %
        %

\end{document}